\documentclass[11pt,a4paper,reqno]{amsart}
\usepackage{amsmath}
\usepackage{amsaddr}
\usepackage{amsfonts}
\usepackage{amssymb}
\theoremstyle{plain}

\usepackage[utf8]{inputenc}
\usepackage[T1]{fontenc}
\usepackage{ae,aecompl}
\usepackage{amsmath}
\usepackage{amssymb}
\usepackage{amsthm}
\usepackage{array}
\usepackage{xy}
\usepackage{bm}
\usepackage{graphicx}
\date{}

\DeclareMathOperator{\lin}{lin}

\DeclareMathOperator{\id}{Id}
\DeclareMathOperator{\dist}{dist}
\newtheorem{twr}{Theorem}[section]
\newtheorem{lem}[twr]{Lemma}

\newtheorem{cor}[twr]{Corollary}
\newtheorem{problem}[twr]{Problem}
\theoremstyle{remark}

\makeatletter
\renewcommand\@seccntformat[1]{\csname the#1\endcsname.\quad}

\makeatother

\begin{document}
\pagestyle{plain}
\title{A uniform estimate of the relative projection constant}

\author{Tomasz Kobos}

\address{Faculty of Mathematics and Computer Science \\ Jagiellonian University \\ Lojasiewicza 6, 30-348 Krakow, Poland}

\email{Tomasz.Kobos@im.uj.edu.pl}


\begin{abstract}
The main goal of the paper is to provide a quantitative lower bound greater than $1$ for the relative projection constant $\lambda(Y, X)$, where $X$ is a subspace of $\ell_{2p}^m$ space and $Y \subset X$ is an arbitrary hyperplane. As a consequence, we establish that for every integer $n \geq 4$ there exists an $n$-dimensional normed space $X$ such that for an every hyperplane $Y$ and every projection $P:X \to Y$ the inequality $||P|| > 1 + \left (8 \left ( n + 3 \right )^{5} \right )^{-30(n+3)^2}$ holds. This gives a non-trivial lower bound in a variation of problem proposed by Bosznay and Garay in $1986$.
\end{abstract}

\maketitle

\section{Introduction}

Let $X$ be a real Banach space and $Y$ its closed subspace. We say that a linear bounded operator $P:X \to Y$ is a \emph{projection} if $P|_{Y}=\id_{Y} $. Let us denote the set of all projections from $X$ onto $Y$ by $\mathcal{P}(X, Y)$. The \emph{relative projection constant} of $Y$ is defined as
$$\lambda(Y, X) = \inf \{ ||P|| : P \in \mathcal{P}(X, Y) \}.$$
Moreover, if a projection $P: X \to Y$ satisfies $||P|| = \lambda(Y, X)$ then $P$ is called a \emph{minimal projection}.

The theory of projection constants and minimal projections has been an important field of research in functional analysis and approximation theory for several decades. Large part of this investigation has focused on the equality $\lambda(Y, X) = 1$, i.e. when there exists a projection $P: X \to Y$ of norm $1$. In such a situation we say that $Y$ is a \emph{one-complemented} subspace of $X$.

One-complemented subspaces of classical Banach spaces have been studied intensively by many authors -- see for example: \cite{ando}, \cite{douglas}, \cite{bohnenblust}, \cite{lewickiorlicz}, \cite{lewickitrombetta}, \cite{baronti1}, \cite{baronti2}, \cite{rand1}, \cite{rand2}, \cite{rand3}, \cite{lemmens}, \cite{tzafriri}. See also \cite{randsurvey} for a survey on this topic. In the setting of $n$-dimensional normed spaces, most spaces actually do not posess any nontrivial one-complemented subspaces. Bosznay and Garay in 1986 (see \cite{bosznay}) have proved that if isometric classes of $n$-dimensional normed spaces are made into the metric space (called \emph{Banach-Mazur compactum}) then the set of spaces without non-trivial one-complemented subspaces is open and dense. In other words, for a general normed space $X$ of dimension $n$ we have $\lambda(Y, X)>1$ for every subspace $Y$ such that $2 \leq \dim Y \leq \dim X - 1$. Therefore, a natural question comes to mind: how far can minimum of relative projection constants deviate from $1$? Formally, Problem $2$ from the paper \cite{bosznay} of Bosznay and Garay asks about finding upper and lower bounds for $\sup_{X}  \inf_{Y \subset X} \lambda(Y, X)$, where $X$ is a real $n$-dimensional normed space and $Y \subset X$ is a subspace of dimension at least $2$ and at most $n-1$. We feel that this is a fascinating problem of a general theory of projections which has not received an adequate attention and can be a fruitful area of further research. To this day, the only results in this direction that are known to author are presented in \cite{kobos} and are concerned only with the upper bounds.

The aim of this paper is to provide a construction of a class of $n$-dimensional normed spaces, for which every hyperplane has a relative projection constant greater then $1 + \varepsilon_0$ for some explicit $\varepsilon_0$. We work therefore with a variant of a problem posed by Bosznay and Garay, concerned only with projections onto hyperplanes.

Such a class of $n$-dimensional normed spaces is actually known for a much longer time. Bohnenblust in 1941 (see \cite{bohnenblust}) proved that a typical subspace of space $\ell_p^m$ with appropriately large codimension usually satisfies such a condition. Let us recall that the $\ell_p^m$ space, where $m \geq 1$ is an integer and $p \geq 1$ is real number, is defined as the normed space $(\mathbb{R}^m, || \cdot ||_p)$ with
$$||x||_p = (|x_1|^p + |x_2|^p + \ldots + |x_m|^p)^{\frac{1}{p}}.$$
Bohnenblust showed that there are no one-complemented subspaces, but did not provide any explicit lower bound for relative projection constant that is greater than $1$. Our goal is to establish such a lower bound in the similar class of normed spaces. We will consider subspaces of the $\ell_{2p}^m$ space of codimension at least $2$ with $p$ being a positive integer. Lower bound on the relative projection constant depends on $p$, $m$, codimension and on a subspace. Precisely we prove the following

\begin{twr}
\label{glowne}
Let $n \geq 4$, $p \geq \frac{m}{2}$ and $m \geq n+2$ be integer numbers. Suppose that $f_1, f_2, \ldots, f_{m}: \mathbb{R}^n \to \mathbb{R}$ are non-zero functionals. Consider a normed space $X= (\mathbb{R}^n, || \cdot ||)$ with the norm defined as
$$||x|| = \left ( \sum_{i=1}^{m} |f_i(x)|^{2p} \right )^{\frac{1}{2p}}.$$
Let $ 0 < \alpha \leq \frac{1}{2}$ be a real number such that for every $0 \leq j < k < l \leq m$ and $0 \leq i \leq m$, $i \not \in \{j, k, l\}$ we have
$$\dist(f_i, \lin\{f_j, f_k, f_l\}) \geq \alpha.$$
Let $\beta>0$ be real number such that for every $0 \leq j < k \leq m$ and $x \in \mathbb{R}^n$ we have
$$\max\{|f_j(x)|, |f_k(x)|\} \leq \beta \max_{1 \leq i \leq m, i \not \in \{j, k\}} |f_i(x)|.$$
Then, for every $(n-1)$-dimensional subspace $Y \subset X$ we have $\lambda(Y, X) > 1 + \varepsilon_0$, where
$$\varepsilon_0 = \varepsilon_0(n, p, m, \alpha, \beta) =  \left ( m+2\beta^{2p} \right )^{-7} \left ( \alpha^{-6} 2^{14} n^{3} m^{11} p^{4} \right )^{-12pm}. $$
\end{twr}

The distance in the definition of parameter $\alpha$ is measured with respect to the norm $|| \cdot ||^{\star}$, dual to $|| \cdot ||$ which is defined by the functionals $f_i$. Note also that our construction does not work for $n=3$.

An application of Theorem \ref{glowne} for a certain choice of functionals $f_i$'s gives us the following
\begin{cor}
\label{wniosek}
For every integer $n \geq 4$ there exists an $n$-dimensional normed space $X$ such that 
$$\lambda(Y, X) > 1 + \left (8 \left ( n + 3 \right )^{5} \right )^{-30(n+3)^2} > 1 + \exp(-C n^2 \log n)$$
for an arbitrary $(n-1)$-dimensional subspace $Y$ of $X$ ($C>0$ is an absolute constant).
\end{cor}

The lower bounds presented above are probably very far from being optimal. However, in spite of the lack of any progress in the problem proposed by Bosznay and Garay and in the further development of the example provided by Bohnenblust, we believe that such a lower bounds might still be interesting. We also hope that our results can bring some attention to the problems of this category and much more efficient methods could be developed in consequence. Moreover, we feel that certain parts of our reasoning may be of independent interest. In Section \ref{sekcjalemat} we prove Lemma \ref{projekcja} which potentially might be a useful tool for providing lower bounds on the relative projection constants. In Section \ref{sekcjafunkcjonaly} we discuss some general problem about linear functionals, which seems to be an interesting open problem of discrete geometry and can turn out to be a fruitful research area. Sections \ref{sekcjaglowne} and \ref{sekcjawniosek} are devoted for proving Theorem \ref{glowne} and Corollary \ref{wniosek} respectively. In general, our approach is elementary. In the last section of the paper we discuss several directions for a possibility of further research.

It is important to note that in the asymptotic setting there are some remarkable results concerning existence of spaces with large relative projection constants. Gluskin in \cite{gluskin} and Szarek in \cite{szarek} have used probabilistic constructions to prove that there are $n$-dimensional normed spaces for which every subspace $Y$ of dimension $m$ in the interval of the form $[\alpha n, \beta n]$ has relative projection constant of order $c \sqrt{m}$ or similar. Both papers contain several results of this type. See also \cite{latala} for a similar construction. Even if these results are very deep, they do not yield any quantitative lower bounds in our problems and none of them touches the case of hyperplanes. Neverthless, they give an important insight and leave a hope that lower bounds obtained in our paper can be improved significantly.

\section{Lemma about projections with small norms}
\label{sekcjalemat}

Let $X$ be a Banach space. It is easy to see that every projection $P:X \to Y$, where $Y=\ker f$ is a hyperplane, can be represented in the form $P(x)=x-f(x)w$, for some $w \in X$ satisfying $f(w)=1$. Let us also recall that if $x \in X$ is nonzero then every continuous linear functional $f:X \to \mathbb{R}$ such that $||f||=1$ and $f(x)=||x||$ is called a \emph{supporting functional} of $x$. By the Hanh-Banach Theorem every nonzero element has at least one supporting functional. If every nonzero vector $x \in X$ has the unique supporting functional, then we say that the Banach space $X$ is \emph{smooth}. In the study of one-complemented hyperplanes the following simple lemma is often crucial (see e.g. \cite{bohnenblust}, \cite{lewickiorlicz}).

\begin{lem}
Let $X$ be a smooth Banach space and let $Y = \ker f$ be a hyperplane in $X$. Suppose that $P:X \to Y$, where $P(x) = x - f(x)w$ and $f(w)=1$ is a projection of norm $1$. Then $f_y(w)=0$ for every nonzero $y \in Y$, where $f_y$ is the unique supporting functional of $y$.
\end{lem}

To study projections of small norm we shall need an extension of this lemma, which gives an upper bound for the value $|f_y(w)|$. It is natural to suspect that quality of such an upper bound should depend on the quality of smoothness of $X$, which is connected to the convexity of the dual space $X^{\star}$. Therefore to state our result, we shall use the \emph{modulus of convexity} of the space $X^{\star}$. Let us recall that for a general Banach space $X$ the modulus convexity $\delta_X: [0, 2] \to \mathbb{R}$ is defined as
$$\delta_X(t) = \inf \left \{1 - \left | \left | \frac{x+y}{2} \right | \right | \ : \ ||x||, \ ||y|| \leq 1 \text{ and } ||x-y|| \geq t \right \}.$$

We have the following
\begin{lem}
\label{projekcja}
Let $X$ be a smooth Banach space and let $Y = \ker f$ be a hyperplane of $X$, where $f \in S_{X^{\star}}$. Suppose that $P:X \to Y$ is a projection of norm not greater than $1+r$, where $P(x) = x - f(x)w$ for some $w$ satisfying $f(w)=1$ and $r \geq 0$. Let $t_0 \in [0, 2]$ be such a number that $\delta_{X^{\star}}(t_0) \geq \frac{r}{2+2r}  $. Then $|f_y(w)| \leq t_0(2+r)$ for every nonzero $y \in Y$.
\end{lem}

\begin{proof}
It is enough to consider vectors $y$ of norm $1$. Let us therefore fix unit vector $y \in Y$ and consider the functional $g = f_y \circ P$. Obviously $g(y)=1$ and $||g|| \leq 1+r$. Hence
$$\left | \left |f_y+\frac{g}{1+r}\right | \right | \geq \left | f_y(y) + \frac{g(y)}{1+r} \right | = \frac{2+r}{1+r}.$$
On the other hand
$$\left | \left |f_y+\frac{g}{1+r}\right | \right | \leq 2 - 2\delta_{X^{\star}}\left ( \left | \left |f_y - \frac{g}{1+r}\right | \right | \right ).$$
Consequently
$$\delta_{X^{\star}}\left ( \left | \left |f_y - \frac{g}{1+r}\right | \right | \right ) \leq \frac{r}{2+2r}\leq \delta_{X^{\star}}(t_0),$$
and therefore $\left | \left |f_y  - \frac{g}{1+r}\right | \right | \leq t_0$ as the modulus of convexity is non-decreasing.

It follows that
$$|f_y(w)|=\left | f_y(w)-\frac{g(w)}{1+r} \right |\leq t_0 \cdot ||w||.$$
To reach the conclusion it is therefore enough to bound the norm of $w$. Fix $\varepsilon>0$ and let $x_0$ be unit vector such that $f(x_0) \geq 1 - \varepsilon$. Then
$$(1-\varepsilon)||w||-1 \leq ||x_0 - f(x_0)w||=||P(x_0)|| \leq 1+r.$$
Since $\varepsilon$ can be arbitrary small we have $||w|| \leq 2+r$ and the proof is finished. 
\end{proof}

Note that the proof works for an arbitrary smooth Banach $X$, although we shall use it only in the finite-dimensional setting. We believe that the lemma above may have some potential for providing lower bounds of the relative projection constant, when one knows something about the modulus of convexity of the dual space and the form of the supporting functionals.

\section{Estimating the max-min of functionals}
\label{sekcjafunkcjonaly}

Let $|| \cdot ||$ be a norm in $\mathbb{R}^n$ and suppose that some collection of $m$ norm-one (in the dual norm of $|| \cdot ||$)  functionals $f_1, f_2, \ldots, f_m$ is given. It is then natural to ask about estimations on the quantity $\max_{||x||=1} \min_{1 \leq i \leq m} |f_i(x)|$. We believe that such a problem could already be investigated, at least in the case of the Euclidean norm. Nevertheless, we shall establish lower bound on this quantity, as we have not found any informations concerning this kind of problem. Our approach is based on measure estimations. We start with

\begin{lem}
\label{objetosc}
Let $n \geq 4$ be an integer. Suppose that the unit $(n-1)$-sphere $\mathbb{S}^{n-1}$ of $\mathbb{R}^{n}$ is equipped with the normalized Lebesgue measure $\mu$. Then for every norm-one functional $f:\mathbb{R}^n \to \mathbb{R}$ and $t \in [0, 1]$ the measure of the set
$$S=\{x: x \in \mathbb{S}^{n-1} \text{ and } |f(x)| \leq t \}$$
is less than $t\sqrt{n}$.
\end{lem}
\begin{proof}
Let $A_k(r)$ denote the surface area of the $k$-sphere in $\mathbb{R}^{k+1}$ of radius $r$ calculated in the usual way. Then it is easy to see that
$$\mu(S)=\frac{2}{A_{n-1}(1)} \int_{\arccos{t}}^{1}  A_{n-2}\left ( \sin \alpha \right)\,d\alpha = \frac{2A_{n-2}(1)}{A_{n-1}(1)} \int_{\arccos{t}}^{1} \left ( \sin \alpha \right)^{n-2}\,d\alpha$$
$$= \frac{2A_{n-2}(1)}{A_{n-1}(1)} \int_{0}^{t} \left (1-u^2 \right )^{\frac{n-3}{2}}\,du \leq \frac{2A_{n-2}(1)}{A_{n-1}(1)} \int_{0}^{t} 1\,du \leq \frac{2tA_{n-2}(1)}{A_{n-1}(1)}.$$
We shall now upper bound the ratio $\frac{A_{n-2}(1)}{A_{n-1}(1)}$ with the help of closed forms for $A_k(r)$ and Stirling's approximation formula. In version of Robbins (see \cite{robbins}) it states that for every positive integer $m$ the following inequalities are true:
\begin{equation}
\label{stirling}
\sqrt{2 \pi} m^{m + \frac{1}{2}} e^{-m} e^{\frac{1}{12m+1}} \leq m! \leq \sqrt{2 \pi} m^{m + \frac{1}{2}} e^{-m} e^{\frac{1}{12m}}.
\end{equation}
We assert that 
$$\frac{A_{n-2}(1)}{A_{n-1}(1)} \leq \frac{\sqrt{n}}{2}.$$ 
Suppose that $n=2k+1$ is an odd number. Then $k \geq 2$ and
$$\frac{A_{n-2}(1)}{A_{n-1}(1)}=\frac{A_{2k-1}(1)}{A_{2k}(1)}=\frac{(2k-1)!}{2^{2k-1}((k-1)!)^2}.$$
By estimations (\ref{stirling}) and easily verified inequality $(1+\frac{1}{m})^{m+1} < \sqrt{\frac{\pi}{2}} e$ (for $m \geq 2$) we have

$$\frac{(2k-1)!}{2^{2k-1}((k-1)!)^2} \leq \frac{(2k-1)^{2k-\frac{1}{2}} \cdot e^{-2k+1} \cdot e^{\frac{1}{12(2k-1)}}}{\sqrt{2 \pi} \cdot 2^{2k-1} \cdot (k-1)^{2k-1} \cdot e^{-2(k-1)} \cdot e^{\frac{2}{12(k-1)+1}}}$$
$$= \sqrt{2k-1} \cdot \frac{1}{\sqrt{2 \pi} e} \cdot \frac{e^{\frac{1}{12(2k-1)}}}{e^{\frac{2}{12(k-1)+1}}}  \left ( \frac{2k-1}{2k-2} \right )^{2k-1}$$ 
$$=\sqrt{2k-1} \cdot \frac{1}{\sqrt{2 \pi} e} \cdot \frac{e^{\frac{1}{12(2k-1)}}}{e^{\frac{2}{12(k-1)+1}}} \cdot  \left ( 1+\frac{1}{2k-2} \right )^{2k-1} < \frac{1}{\sqrt{2 \pi} e} \cdot 1 \cdot  \sqrt{\frac{\pi}{2}} e$$
$$=\frac{1}{2} \sqrt{2k-1} = \frac{1}{2} \sqrt{n}.$$ 

Now we shall consider the case $n=2k$. We have
$$\frac{A_{n-2}(1)}{A_{n-1}(1)}=\frac{A_{2k-2}(1)}{A_{2k-1}(1)}=\frac{2^{2k-3} \cdot (k-2)! \cdot (k-1)!}{\pi (2k-3)!}.$$
For $k=2$ our assertion follows easily. For $k \geq 3$ we apply the Stirling's approximation (\ref{stirling}) again and a simple estimation $e^{\frac{1}{8}} < \sqrt{\frac{\pi}{2}}$ to get
$$\frac{2^{2k-3} \cdot (k-2)! \cdot (k-1)!}{\pi (2k-3)!} \leq \frac{\sqrt{2} \cdot 2^{2k-3} \cdot (k-2)^{k-\frac{3}{2}} \cdot (k-1)^{k-\frac{1}{2}} \cdot e^{2k-3} \cdot e^{\frac{1}{12(k-2)} + \frac{1}{12(k-1)}}}{\sqrt{\pi} \cdot  (2k-3)^{2k-\frac{5}{2}}e^{2k-3}e^{\frac{1}{12(2k-3)+1}}}$$
$$=\frac{\sqrt{2}}{4\sqrt{\pi}} \cdot \frac{(2k-2)^{k-\frac{1}{2}} \cdot (2k-4)^{k-\frac{1}{2}}}{(2k-3)^{2k-1}} \cdot \frac{(2k-3)^{\frac{3}{2}}}{k-2} \cdot e^{\frac{1}{12(k-2)} + \frac{1}{12(k-1)} - \frac{1}{12(2k-3)+1}} $$
$$=\frac{\sqrt{2}}{4\sqrt{\pi}} \cdot \left ( \frac{(2k-3)^2-1}{(2k-3)^2} \right )^{k-\frac{1}{2}} \cdot  \frac{(2k-3)^{\frac{3}{2}}}{k-2} \cdot  e^{\frac{1}{12(k-2)} + \frac{1}{12(k-1)} - \frac{1}{12(2k-3)+1}}$$
$$ \leq \frac{\sqrt{2}}{4\sqrt{\pi}} \cdot 1 \cdot \frac{(2k-3)^{\frac{3}{2}}}{k-2} \cdot e^{\frac{1}{8}} < \frac{\sqrt{2}}{4\sqrt{\pi}} \cdot \frac{(2k-3)^{\frac{3}{2}}}{k-2} \cdot \sqrt{\frac{\pi}{2}} = \frac{(2k-3)^{\frac{3}{2}}}{4(k-2)}.$$

By using an inequality $(2k-3)^3 \leq 8k(k-2)^2$ that can be checked by hand for $k \geq 3$ we conclude finally that 
$$ \frac{(2k-3)^{\frac{3}{2}}}{4(k-2)} \leq \frac{\sqrt{2k}}{2} = \frac{\sqrt{n}}{2}.$$
Thus
$$\mu(S) \leq \frac{2tA_{n-2}(1)}{A_{n-1}(1)} \leq t \sqrt{n}$$ 
and the lemma is proved.

\end{proof}

Main estimate of this section is given by

\begin{lem}
\label{funkcjonaly}
Let $|| \cdot ||$ be a norm in $\mathbb{R}^n$ (where $n \geq 4$) and let $f_1, f_2, \ldots, f_m$ be nonzero functionals. Then, there exists $y \in \mathbb{R}^n$ such that $||y||=1$ and
$$|f_i(y)| \geq \frac{||f_i||}{nm},$$
for every $i=1, 2, \ldots, m$. 
\end{lem}

\begin{proof}
By rescaling we can assume that $||f_i||=1$ for every $1 \leq i \leq m$. First suppose that $|| \cdot || = || \cdot ||_2$ is the Euclidean norm. By Lemma \ref{objetosc} for
$$S_i=\left \{x: ||x||_2=1 \text{ and } |f_i(x)| \leq \frac{1}{\sqrt{n}m} \right \}$$
we have $\mu(S_i) < \frac{1}{m}.$
In consequence
$$\mu(S_1 \cup S_2 \cup \ldots \cup S_m) \leq \mu(S_1) + \mu(S_2) \ldots + \mu(S_m) < \frac{m}{m}=1.$$
It follows that there exists $y \in \mathbb{R}^n$ such that $||y||_2=1$ and $|f_i(y)| \geq \frac{1}{\sqrt{n}m}$ for every $i=1, 2, \ldots, m$.

Suppose now that $|| \cdot ||$ is an arbitrary norm in $\mathbb{R}^n$. By the John Ellipsoid Theorem there exists a linear transformation $T: \mathbb{R}^n \to \mathbb{R}^n$ such that
$$||x||_2 \leq ||Tx|| \leq \sqrt{n} ||x||_2$$
for any $x \in \mathbb{R}^n$. Let $\tilde{f_i} = f_i \circ T$ for $i=1, 2, \ldots, m$. It is easy to check that $||\tilde{f}||_2 \geq 1$. Indeed, consider $x_0$ satisfying $||x_0||=1$ and $|f_i(x_0)| = 1$. Then
$$||T^{-1}(x_0)||_2 \leq ||T(T^{-1}(x_0))|| = ||x_0|| = 1$$
and
$$|\tilde{f_i}(T^{-1}(x_0))| = |f_i(x_0)| = 1.$$
In consequence, we can apply the previous part to the $\tilde{f_i}$'s considered in the Euclidean norm. It yields an existence of $y$ such that $||y||_2 =1$ and
$$|f_i(Ty)| = |\tilde{f_i}(y)| \geq  \frac{1}{\sqrt{n}m}.$$
However, $||T(y)|| \leq \sqrt{n}$ and therefore after an appropriate rescaling the vector $T(y)$ satisfies the desired conditions.

\end{proof}

\section{Proof of Theorem \ref{glowne}}
\label{sekcjaglowne}

In this section we prove Theorem \ref{glowne}. To make use of Lemma \ref{projekcja} we need some information about the modulus of convexity of the dual of a subspace of $\ell_{2p}^{m}$ space. We take care of that in the two following lemmas. Note that in fact we need estimation on the modulus of convexity of a quotient space of $\ell_{q}^{m}$, where $q = \frac{2p}{2p-1}$.

\begin{lem}
\label{modul}
Let $1 \leq q \leq 2$. Then the modulus of convexity of the space $\ell_q^n$ satisfies $\delta_{\ell_q^n}(t) \geq \frac{q-1}{8}t^2$ for every $t \in [0, 2]$.
\end{lem}

\begin{proof}

See \cite{meir}.
\end{proof}

The next lemma basically says that the operation of taking a quotient does not worsen the convexity. 

\begin{lem}
\label{modul2}
Let $X$ be a finite dimensional Banach space and $Y$ its subspace. Then $\delta_{X/Y}(t) \geq \delta_X(t)$ for every $t \in [0, 2]$.
\end{lem}
\begin{proof}
Let us recall that norm $||[x]||_{X/Y}$ in the quotient space is defined as $||[x]||_{X/Y}=\dist(x, Y)$. For every $x \in X$ we clearly have $||[x]||_{X/Y} \leq ||x||$. Let us fix $t \in [0, 2]$ and $\varepsilon>0$. Choose $x,\ y \in X$ such that $||[x]||_{X/Y}, ||[y]||_{X/Y} \leq 1$, $||[x-y]||_{X/Y} \geq t$ and
$$\delta_{X/Y}(t) \geq 1 - \left | \left | \left [ \frac{x+y}{2} \right ] \right | \right | - \varepsilon.$$
All of these inequalities are not changed if we replace $x$ and $y$ by $x-x_1$ and $y-y_1$ respectively, where $x_1,\ y_1 \in Y$ satisfy $\dist(x, Y) = ||x-x_1||,\ \dist(y, Y)=||y-y_1||$ (such $x_1,\ y_1$ exists because of finite dimension). Therefore we can assume that $||x|| \leq 1$ and $||y|| \leq 1$. Then we also have $||x-y|| \geq ||[x-y]||_{X/Y} \geq t.$ Moreover
$$\delta_X(t) \leq 1 - \left | \left | \frac{x+y}{2} \right | \right | \leq 1 - \left | \left | \left [ \frac{x+y}{2} \right ] \right | \right | \leq \delta_{X/Y}(t) + \varepsilon.$$
Since $\varepsilon$ can be arbitrarly small it follows that $\delta_{X/Y}(t) \geq \delta_X(t)$ for every $t \in [0, 2]$ and the proof is finished.
\end{proof}

We need also a formula for a supporting functional in the case of a subspace of $\ell_{2p}^m$ space. It is given in the next lemma.

\begin{lem}
\label{supporting}
Let $X=(\mathbb{R}^n, || \cdot ||)$ be the normed space defined in Theorem \ref{glowne}. Let $y \in X$ be nonzero vector. Then the supporting functional $f_y$ of vector $y$ is given by
$$f_y(x) = \frac{1}{||y||^{2p-1}}\sum_{i=1}^{m} f_i(y)^{2p-1}f_i(x).$$
\end{lem}
\begin{proof}
Obviously $f_y(y)=||y||$ and it sufficies to check that $||f_y|| \leq 1$. But this follows directly from H\"older's inequality.
\end{proof}

The well-known characterization of one-complemented subspaces of classical $\ell_p^{m}$ spaces (see for example \cite{lewickiorlicz} for much more general result) states that $\lambda(\ker f, \ell_{p}^{m})=1$ if and only if the vector corresponding to a functional $f$ has at most two coordinates that are different from $0$. In other words, if we denote be $e(i)$ the unit vectors from the canonical basis then $\lambda(\ker f, \ell_{p}^{m})=1$ if and only if $f = a e(i)$ for some $1 \leq i \leq m$, $a \neq 0$ or $f = ae(i) + be(j)$ for $1 \leq i < j \leq m$ and $a, b \neq 0$. In our setting we have corresponding situations in which functional $f$ is close to some functional of the form $a f_i$ or $af_i + bf_j$ (where $f_1, f_2, \ldots, f_m$ are functionals defining the subspace). It turns out that in these cases the relative projection constant is still greater than $1$, but some special treatment is necessary. We shall thus consider three cases: functional $f$ is close to a functional of the form $af_i$, functional $f$ is close to a functional of the form $af_i + bf_j$ and neither of these. Although reasoning in each of these possibilities runs along similar lines, there are some adjustments necessary to fit the argument to each situation. In fact, much of the difficulty of the proof of Theorem \ref{glowne} is hidden in a careful choice of the precise range in which we say that $f$ is ,,close'' to $af_i$ or $af_i + bf_j$. It is crucial to know that $f$ can not be close to two functionals of this form at the same time. We establish this type of result in the two following lemmas.

\begin{lem}
\label{blisko1}
Let be $|| \cdot ||$ be an arbitrary norm in $\mathbb{R}^n$ and let $f, f_1, \ldots, f_m \in \mathbb{R}^n$. Assume that $ 0 < \alpha \leq \frac{1}{2}$ is a real number such that for every $0 \leq j < k < l \leq m$ and $0 \leq i \leq m$, $i \not \in \{j, k, l\}$ we have
$$\dist\left ( f_i, \lin\{f_j, f_k, f_l\} \right ) \geq \alpha,$$ 
where the distance is with respect to the norm $|| \cdot ||$. Suppose that there exist indices $1 \leq k, l \leq m$, $k \neq l$ such that $||f_k+ a_0f_l + r_0f|| \leq \frac{\alpha}{2}$ for some $a_0, r_0 \in \mathbb{R}$. Then $||f_i+af_j +rf|| \geq \frac{\alpha}{2}$ for every $1\leq i, j \leq m$, $i \not \in \{j, k,\ l\}$, $j \neq k$ and $a, r \in \mathbb{R}$.

\end{lem}

\begin{proof}
Assume that for some $i,\ j,\ a,\ r$ as above the opposite inequality is true. It is clear that $r,\ r_0$ are nonzero and to reach contradiction we can suppose that $|r_0| \geq |r|$, as the conditions are now symmetric. It follows that
$$||f_i+af_j +rf||=\left |\left |f_i + af_j + \frac{r}{r_0}\left ( f_k + a_0f_l + r_0f \right ) - \frac{r}{r_0}\left(f_k + a_0f_l \right ) \right | \right|$$
$$= \left | \left |\left (f_i + af_j - \frac{r}{r_0}f_k - \frac{a_0r}{r_0} f_l \right ) + \frac{r}{r_0}\left ( f_k+a_0f_l +r_0f \right ) \right | \right | \geq \alpha - \left | \frac{r}{2r_0}  \right | \alpha \geq \frac{\alpha}{2}$$
This is a contradiction with the assumption and the lemma is proved.
\end{proof}

The result above does not cover the case $i=l$, which shall be treated in the next lemma.

\begin{lem}
\label{blisko2}
Let be $|| \cdot ||$ be an arbitrary norm in $\mathbb{R}^n$ and let $f, f_1, \ldots, f_m \in \mathbb{R}^n$. Assume that $ 0 < \alpha \leq \frac{1}{2}$ is a real number such that for every $0 \leq j < k < l \leq m$ and $0 \leq i \leq m$, $i \not \in \{j, k, l\}$ we have
$$\dist\left ( f_i, \lin\{f_j, f_k, f_l\} \right ) \geq \alpha,$$ 
where the distance is with respect to the norm $|| \cdot ||$. Suppose moreover that $0 < L < K < \frac{1}{2}$ are real numbers such that $K\alpha> 4L$. Assume that there exist indices $1 \leq k, l \leq m$, $k \neq l$ such that $||f_k+ a_0f_l + r_0f|| \leq L$ for some $a_0, r_0 \in \mathbb{R}$ and $||f_l + rf|| \geq K$ for every $r \in \mathbb{R}$. Then $||f_i+af_j +rf|| \geq \frac{K\alpha}{2}$ for every $i,\ j \in \{1, 2, \ldots, m\} \setminus \{k\}$, $i \neq j$ and $a, r \in \mathbb{R}$.
\end{lem}

\begin{proof}

By Lemma \ref{blisko1} it is enough to consider the case $i=l$. Suppose that $||f_l+af_j+rf|| < \frac{K\alpha}{2}$. Then
$$\frac{K\alpha}{2} > ||(f_l + rf)+af_j|| \geq K - |a|$$
and therefore $|a| \geq K\left ( 1-\frac{\alpha}{2} \right ).$

Assume that $|r| \geq |r_0|$. We obtain
$$||f_l + af_j + rf|| = \left | \left |f_l + af_j + \frac{r}{r_0} \left ( f_k+ a_0f_l + r_0f \right) - \frac{r}{r_0}f_k - \frac{a_0r}{r_0}f_l \right | \right |$$
$$\geq \left | \frac{r}{r_0} \right |\alpha - \left | \frac{r}{r_0} \right |L \geq \alpha - L > \alpha - \frac{K\alpha}{2} - \frac{K\alpha^2}{2}\geq  \frac{K\alpha}{2}$$
which contradicts our assumption. Hence $|r| < |r_0|$. We can estimate similarly like before to get
$$||f_l + af_j + rf|| = \left | \left |f_l + af_j + \frac{r}{r_0} \left ( f_k+ a_0f_l + r_0f \right) - \frac{r}{r_0}f_k - \frac{a_0r}{r_0}f_l \right | \right |$$
$$\geq |a| \alpha - \left | \frac{r}{r_0} \right | L \geq K \left ( \alpha-\frac{\alpha^2}{2} \right ) - L > \frac{K\alpha}{2}.$$
We have again reached a contradiction, which finishes the proof of the lemma.

\end{proof}

One of the key ingredients in the original reasoning of Bohnenblust in \cite{bohnenblust} was the invertibility of the Vandermonde matrix. For our purposes we need some quantitative version of this result. We shall use the following estimation due to Gautschi. For a matrix $A$ of dimensions $m \times m$ we consider its norm as of an operator $A: \ell_{\infty}^{m} \to \ell_{\infty}^m$, that is $||Ax|| = \sup_{||x||_{\infty} \leq 1} ||Ax||_{\infty}$.

\begin{lem}
\label{vandermonde}
Let $x_1, x_2, \ldots, x_m$ be pairwise distinct real numbers and let $V$ be the Vandermonde matrix with columns of the form $(1, x_i, x_i^2, \ldots, x_i^{m-1})$. Then 
$$||V^{-1}|| \leq \max_{1 \leq i \leq m} \prod_{j \neq i} \frac{1+|x_j|}{|x_j-x_i|}.$$
\end{lem}
\begin{proof}
See \cite{gautschi}.
\end{proof}

Before giving a proof of Theorem \ref{glowne} we need a last small observation.

\begin{lem}
\label{zawezenie}
Let $X$ be a Banach space. Suppose that $f,\ g: X \to \mathbb{R}$ are two linear functionals such that $||f-rg|| \geq a$ for every $r \in \mathbb{R}$ and some $a \in \mathbb{R}$. Then  $||f|_{\ker g}|| \geq a$. 
\end{lem}
\begin{proof}
By the Hahn-Banach Theorem there exists a linear functional $\tilde{f}$ whose restriction to $\ker g$ is the same as restriction of $f$ and its norm is equal to $||f|_{\ker g}||$. We can write $f - \tilde{f} = rg$ for some real $r$. Then
$$||f|_{\ker g}||=||\tilde{f}||=||f-rg|| \geq a.$$
\end{proof}

Finally we can move to the proof of our main result.


\emph{Proof of Theorem \ref{glowne}.}
We begin with introducing some notation. Let
$$\varepsilon_1=(m+\beta^{2p}-1)^{\frac{-1}{p}} \alpha^{4p+4m} 2^{-(8p+4m+6)} n^{-(4p+4m)}m^{-(4p+6m)}p^{-(2m+1)},$$
$$R_1 = 8 \sqrt{\varepsilon_1p}, \qquad K = \left ( \frac{R_1}{4} \right )^{\frac{1}{2p-1}},$$
$$\varepsilon_2=K^{2m} (m+2\beta^{2p}-2)^{\frac{-1}{p}} \alpha^{4p+4m} 2^{-(8p+4m+6)} n^{-(4p+4m)}m^{-(4p+6m)}p^{-(2m+1)},$$
$$R_2= 8 \sqrt{\varepsilon_2p}, \qquad L = \frac{R_2}{2^{2p}(2p-1)},$$
$$\varepsilon_3= m^{\frac{-1}{p}} L^{2m} K^{4p+2m}2^{-(4p+6)}n^{-(4p+4m)}m^{-(4p+6m)}p^{-(2m+1)},$$
$$R_3 = 8 \sqrt{\varepsilon_3p}.$$

Let $Y= \ker f$, where $||f||=1$ and suppose that $\lambda(Y, X) = 1 + \varepsilon$. We will show a stronger statement. We shall prove that
\begin{itemize}
\item if there exist $1 \leq k \leq m$ and $r_0 \in \mathbb{R}$ such that $||f_k+r_0f|| \leq K $, then $\varepsilon \geq \varepsilon_1$.
\item If there exists a pair $1 \leq k < l \leq m$ such that $||f_k + a_0f_l + r_0f|| \leq L$ for some $a_0, r_0 \in \mathbb{R}$, but $||f_i+rf|| > K$ for every $1 \leq i \leq m$ and every $r \in \mathbb{R}$, then $\varepsilon \geq \varepsilon_2$.
\item If $||f_i+rf|| > K$ for every $1 \leq i \leq m,\ r \in \mathbb{R}$ and $||f_i + af_j + rf|| > L$ for every $1 \leq i, j \leq m$, $i \neq j$, $a,\ r \in \mathbb{R}$, then $\varepsilon \geq \varepsilon_3$.
\end{itemize}

Conclusion of the theorem will then follow from the inequality $\varepsilon_3 \geq \varepsilon_0$ which can be verified with straightforward but a tedious computation.

Let $P:X \to Y$ be a projection such that $||P|| = \lambda(Y, X) = 1 + \varepsilon$ and suppose that $P(x) = x - f(x)w$ for some $w$ satisfying $f(w)=1$. Fix a nonzero vector $y \in Y$. We shall bound $|f_y(w)|$ in terms of $\varepsilon$, where $f_y$ is the unique functional such that $f_y(y)=||y||$ and $|f_y(x)| \leq ||x||$ for $x \in X$. Precisely, we shall prove that
\begin{equation} \label{funkcjonal}
|f_y(w)| \leq 8\sqrt{\varepsilon p}.
\end{equation}

Indeed, by Lemma \ref{projekcja} we have $|f_y(w)| \leq t_0(2+\varepsilon)$ for any $t_0 \in [0, 2]$ satisfying $\delta_{X^{\star}}(t_0) \geq \frac{\varepsilon}{2+2\varepsilon}$. Note that $X$ is clearly a subspace of $\ell_{2p}^m$ and therefore $X^{\star}$ is a quotient space of $\ell_{q}^{m}=(\ell_{2p}^m)^{\star}$, where $q = \frac{2p}{2p-1}$. Take $t_0 = 4\sqrt{\frac{\varepsilon p}{2+2\varepsilon}}$. If $\varepsilon \leq \varepsilon_1$ then, by looking at the expression defining $\varepsilon_1$, we can easily verify that $t_0 \in [0, 2]$. By combining Lemma \ref{modul} with Lemma \ref{modul2} we get
$$\delta_{X^{\star}}(t_0) = \delta_{X^{\star}} \left ( 4\sqrt{\frac{\varepsilon p}{2+2\varepsilon}} \right ) \geq  \delta_{\ell_q^m} \left ( 4\sqrt{\frac{\varepsilon p}{2+2\varepsilon}} \right ) \geq \frac{q-1}{8} \cdot \frac{16\varepsilon p}{2+2\varepsilon}$$
$$=\frac{2p}{2p-1} \cdot \frac{\varepsilon}{2+2\varepsilon} > \frac{\varepsilon}{2+2\varepsilon}.$$

We can therefore use Lemma \ref{projekcja} to obtain
$$|f_y(w)| \leq t_0(2+\varepsilon) = 4(2+\varepsilon)\sqrt{\frac{\varepsilon p}{2+2\varepsilon}}<4(2+2\varepsilon)\sqrt{\frac{\varepsilon p}{2+2\varepsilon}}$$
$$=4\sqrt{\varepsilon p(2+2\varepsilon)}<4\sqrt{4\varepsilon p}=8\sqrt{\varepsilon p},$$
as claimed.

Now we shall consider seperately each of the cases listed at the beginning of the proof. First suppose that there exist $1 \leq k \leq m$ and $r_0 \in \mathbb{R}$ such that $||f_k+r_0f|| \leq K$. We can assume that $k=m$. For the sake of contradiction let us suppose that $\varepsilon \leq \varepsilon_1$. From \eqref{funkcjonal} it follows that
$$|f_y(w)| \leq 8\sqrt{\varepsilon_1 p}=R_1.$$
Moreover, by Lemma \ref{blisko1} we have $||f_i+af_j +rf|| \geq \frac{\alpha}{2}$ for every $1\leq i, j \leq m-1$, $i \neq j$ and $a, r \in \mathbb{R}$. 
By applying Lemma \ref{funkcjonaly} and Lemma \ref{zawezenie} we can choose $y \in Y, ||y||=1$ such that 
\begin{equation} \label{igrek}
|f_i(y)| \geq \frac{\alpha}{2n(m-1)}
\end{equation}
for $1 \leq i \leq m-1$. Obviously $||f_i|| \leq 1$ for $1 \leq i \leq m-1$ and therefore $|f_i(y)| \leq 1$. Furthermore, since for $1 \leq i < j \leq m-1$ and $r \in \mathbb{R}$ we have
$$\left |\left | f_i - \frac{f_i(y)}{f_j(y)}f_j + rf \right |\right | \geq \frac{\alpha}{2}$$
it follows that
$$\left | \left |\frac{f_i}{f_i(y)} -\frac{f_j}{f_j(y)} + \frac{r}{f_i(y)}f \right | \right| \geq \frac{\alpha}{2|f_i(y)|} \geq \frac{\alpha}{2}.$$ 
Again by Lemma \ref{funkcjonaly} and Lemma \ref{zawezenie}, applied to the functionals of the form $\frac{f_i}{f_i(y)} -\frac{f_j}{f_j(y)}$, we can find $z \in Y,\ ||z||=1$ such that

\begin{equation}
\label{zet}
\left | \frac{f_i(z)}{f_i(y)} - \frac{f_j(z)}{f_j(y)} \right | \geq \frac{\alpha}{n(m-2)(m-1)}
\end{equation}
for every pair $1 \leq i < j \leq m-1$.

Now consider a polynomial $P(t)$ defined as
$$P(t) = \sum_{i=1}^{m-1} (f_i(y+tz))^{2p-1} \cdot f_i(w) = \sum_{i=1}^{m-1} (f_i(y) + tf_i(z))^{2p-1} \cdot f_i(w).$$
By the formula for the supporting functional given in Lemma \ref{supporting} it easily follows that
$$P(t) =  f_{y+tz}(w) \cdot ||y+tz||^{2p-1} - f_m(y+tz)^{2p-1} \cdot f_m(w).$$
By the previous part we have $|f_{y+tz}(w)| \leq R_1.$ Note also that since $||f_m+r_0f|| \leq K$ and $||y+tz|| \leq 2$ for $-1 \leq t \leq 1$ we get  
$$|f_m(y+tz)|=|f_m(y+tz)+r_0f(y+tz)| \leq 2 K.$$
If $x_0$ satisfies $||x_0||=1$ and $f(x_0)=1$, then $||w||=||x_0-P(x_0)|| \leq 2 + \varepsilon < 4$. Therefore, by combininig the estimation above with an observation $|f_m(w)| \leq ||w|| < 4$ we obtain the inequality
$$|P(t)| \leq 2^{2p-1}R_1 + 2^{2p+1}K^{2p-1} = 2^{2p-1}R_1 + 2^{2p-1}R_1 = 2^{2p}R_1.$$
for every $t \in [-1, 1]$. By the Markov inequality,
$$|P^{(k)}(0)| \leq 2^{2p}(2p-1)^2(2p-2)^2 \ldots (2p-k)^2R_1$$
for every $0 \leq k \leq m-2$.
On the other hand, a simple calculation shows that
$$|P^{(k)}(0)|=(2p-1)(2p-2) \ldots (2p-k)\sum_{i=1}^{m-1} f_i(y)^{2p-k-1}f_i(z)^k f_i(w).$$
In particular
$$ \left | \sum_{i=1}^{m-1} f_i(y)^{2p-k-1}f_i(z)^k f_i(w) \right | \leq 2^{2p} (2p)^{m-2} R_1.$$
for every $0 \leq k \leq m-2$.
If we denote by $A$ the Vandermonde matrix of the numbers $\left \{ \frac{f_i(z)}{f_i(y)} \right  \}_{i=1, 2, \ldots, m-1}$ and by $v$ we denote the vector $v=[f_i(y)^{2p-1}f_i(w)]_{i=1, 2, \ldots, m-1}$ then we have $||Av||_{\infty} \leq 2^{2p} (2p)^{m-2} R_1$. On the other hand, we obviously have $||Av||_{\infty} \geq \frac{||v||_{\infty}}{||A^{-1}||}.$ Thus, by using the upper bound on $||A^{-1}||$ given in Lemma \ref{vandermonde} combined with estimations \eqref{igrek} and \eqref{zet}, we obtain 
$$ \left (\frac{\alpha}{2n(m-1)}\right )^{2p-1} \cdot \max_{1 \leq i \leq m-1} |f_i(w)| \leq ||A^{-1}|| \cdot 2^{2p} (2p)^{m-2} R_1$$
$$\leq \left( 1 + \frac{2n(m-1)}{\alpha} \right )^{m-2} \left (\frac{n(m-1)(m-2)}{\alpha} \right )^{m-2} 2^{2p} (2p)^{m-2} R_1 .$$
Finally, from the inequality $1+\frac{2n(m-1)}{\alpha} < \frac{2nm}{\alpha}$ and similar crude upper bounds we conclude that
$$\max_{1 \leq i \leq m-1} |f_i(w)| < \alpha^{-(2p+2m)} 2^{4p+2m} n^{2p+2m}m^{2p+3m}p^m R_1 = (m+\beta^{2p}-1)^{\frac{-1}{2p}}.$$
As $||f||=1$ and $f(w)=1$ it is clear $||w|| \geq 1$. But on the other hand, taking into account the inequality $|f_m(w)| \leq \beta \max_{1 \leq i \leq m-1} |f_i(w)|$, we also have
$$||w||=\left ( \sum_{i=1}^{m} |f_i(w)|^{2p} \right )^{\frac{1}{2p}} < \left ( \frac{m-1}{m+\beta^{2p}-1} + \frac{\beta^{2p}}{m+\beta^{2p}-1} \right )^{\frac{1}{2p}} = 1.$$
We have obtained a contradiction which finishes the proof in the considered case.


Now we shall consider the case in which there exists a pair $1 \leq k < l \leq m$ such that $||f_k + a_0f_l + r_0f|| \leq L$ for some $a_0, r_0 \in \mathbb{R}$, but $||f_i+rf|| > K$ for every $1 \leq i \leq m$ and every $r \in \mathbb{R}$. We may assume that $k=m$ and $l=m-1$. In this case we shall reach a contradiction with an assumption that $\varepsilon \leq \varepsilon_2$. From \eqref{funkcjonal} follows that
$$|f_y(w)| \leq 8 \sqrt{\varepsilon_2p} = R_2.$$

From Lemma \ref{blisko1} it follows that 
$$||f_i + rf|| \geq \frac{\alpha}{2}$$
for $1 \leq i \leq m-1$. Since we also have $||f_{m-1}+rf|| > K$, according to Lemma \ref{funkcjonaly} we can choose $y \in \ker f$, $||y||=1$ such that
\begin{equation}
\label{funkcjonal2}
|f_i(y)| \geq \frac{\alpha}{2n(m-1)} \: (1 \leq i \leq m-2) \: \text{ and } \: |f_{m-1}(y)| \geq \frac{K}{n(m-1)}
\end{equation}
Let $1 \leq i, j \leq m-1$, $i \neq j$ and $a,\ r \in \mathbb{R}$. By Lemma \ref{blisko2} we have 
$$||f_i + af_j + rf|| \geq \frac{K\alpha}{2}$$
and therefore
$$\left |\left | f_i - \frac{f_i(y)}{f_j(y)}f_j + rf \right |\right | \geq \frac{K\alpha}{2},$$
so that
$$\left | \left |\frac{f_i}{f_i(y)} -\frac{f_j}{f_j(y)} + \frac{r}{f_i(y)}f \right | \right| \geq \frac{K\alpha}{2|f_i(y)|} \geq \frac{K\alpha}{2}.$$

Lemma \ref{funkcjonaly} combined with Lemma \ref{zawezenie} yields a vector $z \in \ker f$, $||z||=1$ such that 
\begin{equation}
\label{zet2}
\left | \frac{f_i(z)}{f_i(y)} - \frac{f_j(z)}{f_j(y)} \right | \geq \frac{K\alpha}{nm(m-1)},
\end{equation}
for every $1 \leq i < j \leq m-1$.
Similarly like before we consider the polynomial $P(t)$ defined as
$$P(t) = \sum_{i=1}^{m-1} (f_i(y+tz))^{2p-1} \cdot f_i(w) - a_0^{2p-1}(f_{m-1}(y+tz))^{2p-1} \cdot f_m(w)$$
$$= \sum_{i=1}^{m-1} (f_i(y) + tf_i(z))^{2p-1} \cdot f_i(w) - a_0^{2p-1}(f_{m-1}(y)+tf_{m-1}(z)))^{2p-1} \cdot f_m(w)$$
$$= \sum_{i=1}^{m-2} (f_i(y) + tf_i(z))^{2p-1} \cdot f_i(w) + (f_{m-1}(y) + tf_{m-1}(z))^{2p-1}(f_{m-1}(w) - a_0^{2p-1}f_{m}(w))$$

Note that
$$|f_m(y+tz) + a_0f_{m-1}(y+tz)|=|f_m(y+tz) + a_0f_{m-1}(y+tz)+r_0f(y+tz)| \leq L ||y+tz||.$$
Since $|f_m(y+tz)| \leq ||y+tz||$ we also have $||a_0f_{m-1}(y+tz)|| \leq ||y+tz|| (1 + L).$ Therefore
$$\left | P(t) -f_{y+tz}(w) \cdot ||y+tz||^{2p-1} \right | = \left | a_0^{2p-1}(f_{m-1}(y+tz))^{2p-1} + (f_{m}(y+tz))^{2p-1} \right | \cdot |f_m(w)|$$
$$=\left | (f_m(y+tz)+a_0f_{m-1}(y+tz))\left ( f_m(y+tz)^{2p-2} - \ldots +  a^{2p-2}_0f_{m-1}(y+tz)^{2p-2} \right) \right| \cdot |f_m(w)|$$
$$\leq (2p-1) L(1+L)^{2p-2} |f_m(w) ||y+tz||^{2p-1}.$$
In the previous part we have proved that $||w|| \leq 4$ and hence $|f_m(w)| \leq 4$. Thus we obtain an upper bound on $|P(t)|$ for $t \in [-1, 1]$
$$|P(t)| \leq |f_{y+tz}(w)| \cdot ||y+tz||^{2p-1} +  4(2p-1)L(1+L)^{2p-2} ||y+tz||^{2p-1}$$
$$\leq 2^{2p-1}R_2 + 2^{4p-1}(2p-1) L = 2^{2p-1}R_2 + 2^{2p-1}R_2 = 2^{2p}R_2.$$

Now we can follow the same idea as before of estimating the norm of the inverse of the Vandermonde matrix combined with the inequalities \eqref{funkcjonal2} and \eqref{zet2} to conclude that
$$ \left (\frac{\alpha}{2n(m-1)}\right )^{2p-1} \cdot \max_{1 \leq i \leq m-2} |f_i(w)| \leq ||A^{-1}|| \cdot 2^{2p} (2p)^{m-2} R_2$$
$$\leq \left( 1 + \frac{2n(m-1)}{\alpha} \right )^{m-3} \cdot \left( 1 + \frac{2n(m-1)}{K} \right ) \cdot \left (\frac{n(m-2)(m-1)}{K\alpha} \right )^{m-2} 2^{2p} (2p)^{m-2} R_2.$$
Hence
$$\max_{1 \leq i \leq m-2} |f_i(w)| < K^{-m} \alpha^{-(2p+2m)} 2^{4p+2m} n^{2p+2m}m^{2p+3m}p^m R_2 = (m+2\beta^{2p}-2)^{\frac{-1}{2p}}.$$
Now we can reach a contradiction in the same way as in the previous case as
$$||w||=\left ( \sum_{i=1}^{m} |f_i(w)|^{2p} \right )^{\frac{1}{2p}} < \left ( \frac{m-2}{m+2\beta^{2p}-2} + \frac{2\beta^{2p}}{m+2\beta^{2p}-2} \right )^{\frac{1}{2p}} = 1.$$

We move to the last part of the proof. In the remaining case we assume that $||f_i+rf|| > K$ for every $1 \leq i \leq m,\ r \in \mathbb{R}$ and
$$||f_i + af_j + rf|| > L$$
for every $1 \leq i, j \leq m$, $i \neq j$, $a,\ r \in \mathbb{R}$. For the sake of contradiction we also suppose that $\varepsilon \leq \varepsilon_3$. Then
$$|f_y(w)| \leq 8 \sqrt{\varepsilon_3p} = R_3.$$

Using the same reasoning as before, this time simply to the polynomial
$$P(t) = \sum_{i=1}^{m} f_i(y+tz)^{2p-1} \cdot f_i(w)$$
for normed $y,\ z \in \ker f$ satisfying
$$|f_i(y)| \geq \frac{K}{nm},$$
for $1 \leq i \leq m$ and
$$\left | \frac{f_i(z)}{f_i(y)} - \frac{f_j(z)}{f_j(y)} \right | \geq \frac{2L}{nm(m-1)},$$
for every $1 \leq i < j \leq m$, we easily obtain the inequality
$$ \left (\frac{K}{nm}\right )^{2p-1} \cdot \max_{1 \leq i \leq m} |f_i(w)| \leq \left( 1 + \frac{nm}{K} \right )^{m-1} \left (\frac{nm(m-1)}{2L} \right )^{m-1} 2^{2p-1} (2p)^{m-1} R_3,$$
which gives us
$$\max_{1 \leq i \leq m} |f_i(w)| < L^{-m} K^{-(2p+m)}2^{2p}n^{2p+2m}m^{2p+3m}p^mR_3 = m^{\frac{-1}{2p}}.$$
We can again bound the norm of $w$ to get
$$ 1 \leq ||w|| = \left ( \sum_{i=1}^{m} |f_i(w)|^{2p} \right )^{\frac{1}{2p}} < \left ( \sum_{i=1}^{m} \frac{1}{m} \right )^{\frac{1}{2p}}=1.$$
We have obtained a contradiction that completes the last step of the proof.
\qed

\section{Proof of Corollary \ref{wniosek}}
\label{sekcjawniosek}

In this section we apply Theorem \ref{glowne} to establish Corollary \ref{wniosek}

\emph{Proof of Corollary \ref{wniosek}}
We will use Theorem \ref{glowne} for explicit functionals $f_1, f_2, \ldots, f_m$. Let $X=(\mathbb{R}^n, || \cdot ||)$ where the norm $|| \cdot ||$ is defined as in Theorem \ref{glowne} with $m=n+2$, $p=\lceil \frac{n+2}{2} \rceil$, $f_i(x) = x_i$ for $1 \leq i \leq n$, $f_{n+1}(x) = x_1 + x_2 + \ldots + x_n$ and $f_{n+2}(x) = \frac{x_1 + 2x_2 + \ldots + nx_n}{n}.$ We shall estimate the parameters $\alpha$ and $\beta$ of Theorem \ref{glowne} for such a choice of functionals. It is straightforward to do, albeit requires consideration of many cases.

First we shall prove that $\alpha \geq \frac{1}{2n}$, that is $\dist(f_i, \lin\{f_j, f_k, f_l\}) \geq \frac{1}{2n}$ for every $0 \leq j < k < l \leq n+2$, $0 \leq i \leq n+2$, $i \not \in \{j, k, l\}$. Note that for every vector $v \neq 0$ such that $f_j(v)=f_k(v)=f_l(v) = 0$ we have 
$$\dist(f_i,\ \lin\{f_j, f_k, f_l\}) \geq \frac{|f_i(v)|}{||v||}.$$
For different indices $i,\ j,\ k,\ l$ we shall use different vectors $v$ to get the desired lower bound. Suppose that
\begin{itemize}
\item $i, j, k, l \leq n$. Take $v=e(i)$. Then $||v||=\left ( 2+\frac{i^{2p}}{n^{2p}} \right )^{\frac{1}{2p}} \leq 2$ and $f_i(v)=1$. Therefore the distance is at least $\frac{1}{2}$.
\item $i, j, k \leq n$ and $l=n+1$. As $n \geq 4$ we can pick $s \in \{1, 2, \ldots, n \} \setminus \{i, j, k\}$. Take $v = e(i)-e(s)$. Then $||v|| \leq 4$ and $f_i(v)=1$. The distance is at least $\frac{1}{4}$.
\item $i, j, k \leq n$ and $l=n+2$. Pick $s \in \{1, 2, \ldots, n \} \setminus \{i, j, k\}$ and $v = se(i) - ie(s)$. Then $||v|| = \left ( s^{2p} + i^{2p} + \frac{|s-i|^{2p}}{n^{2p}}\right )^{\frac{1}{2p}} \leq 2n$ and $f_i(v)=s \geq 1$. The distance is at least $\frac{1}{2n}$.
\item $i, j \leq n$, $k=n+1$ and $l=n+2$. Pick distinct $s_1, s_2 \in \{1, 2, \ldots, n \} \setminus \{i, j\}$ and $v = e(i) + \frac{i-s_2}{s_2-s_1}e(s_1) + \frac{s_1-i}{s_2-s_1}e(s_2)$. Then $||v|| = \left ( 1 + \frac{(i-s_1)^{2p} + (i-s_2)^{2p}}{(s_1-s_2)^{2p}} \right )^{\frac{1}{2p}}  \leq \left ( 1 + 2n^{2p} \right )^{\frac{1}{2p}} \leq 2n$ and $f_i(v)=1$. The distance is at least $\frac{1}{2n}$.
\item $i=n+1$, $j, k, l \leq n$. Pick $s \in \{1, 2, \ldots, n \} \setminus \{j, k, l\}$ and $v = e(s)$. Then $||v|| \leq 2$ and $f_i(v)=1$. The distance is at least $\frac{1}{2}$.
\item $i=n+1$, $j, k \leq n$ and $l=n+2$. Pick distinct $s_1, s_2 \in \{1, 2, \ldots, n \} \setminus \{j, k\}$ and $v = s_2e(s_1) - s_1e(s_2)$. Then $||v|| \leq 2n$ and $|f_i(v)|=|s_1-s_2| \geq 1$. The distance is at least $\frac{1}{2n}$.
\item $i=n+2$, $j, k, l \leq n$. Pick $s \in \{1, 2, \ldots, n \} \setminus \{j, k, l\}$ and $v = e(s)$. Then $||v|| \leq 2$ and $f_i(v)=\frac{s}{n} \geq \frac{1}{n}$. The distance is at least $\frac{1}{2n}$.
\item $i=n+2$, $j, k \leq n$ and $l=n+1$. Pick distinct $s_1, s_2 \in \{1, 2, \ldots, n \} \setminus \{j, k\}$ and $v = e(s_1) - e(s_2)$. Then $||v|| \leq 4$ and $|f_i(v)|=\frac{|s^2_1-s^2_2|}{n} \geq \frac{2}{n}$. The distance is at least $\frac{1}{2n}$.
\end{itemize}	

We have thus established that $\alpha \geq \frac{1}{2n}$. In a similar manner we will now upper bound the parameter $\beta$ by $n^2$. In other words, we shall prove that for $0 \leq j < k \leq n+2$ and $x \in \mathbb{R}^n$ we have
$$\max\{|f_j(x)|, |f_k(x)|\} \leq n^2 \max_{1 \leq i \leq m, i \not \in \{j, k\}} |f_i(x)|.$$
We will do this by writing each functional $f_i(x)$ as a linear combination of every $n$ of the remaining ones with the sum of absolute values of coefficients not exceeding $n^2$. In fact, suppose that
\begin{itemize}
\item $j=n+1$ and $k=n+2$. Then $f_{n+1} = \sum_{i=1}^{n} f_i$ and $f_{n+2} = \sum_{i=1}^{n} if_i$.
\item $j \leq n$ and $k=n+2$. Then $f_j = f_{n+1} - \sum_{1 \leq i \leq n, i \neq j} f_i$ and $f_{n+2} =\frac{1}{n} \left ( jf_{n+1} + \sum_{i=1}^{n} (i-j)f_i \right ) $. 
\item $j \leq n$ and $k=n+1$. Then $f_j =\frac{n}{j}f_{n+2} - \sum_{1 \leq i \leq n, i \neq j} \frac{i}{j} f_i$ and $f_{n+1} = \frac{n}{j}f_{n+2} - \sum_{1 \leq i \leq n, i \neq j} \left ( \frac{i}{j} - 1 \right ) f_i$.
\item $j < k \leq n$. Then $f_j = \frac{n}{j} f_{n+2} - \frac{k}{j} f_{n+1} + \sum_{1 \leq i \leq n, i \neq j} \left ( \frac{k}{j} - \frac{i}{j} \right ) f_i$ and similarly $f_k = \frac{n}{k} f_{n+2} - \frac{j}{k} f_{n+1} + \sum_{1 \leq i \leq n, i \neq k} \left ( \frac{j}{k} - \frac{i}{k} \right ) f_i$.
\end{itemize}
It is straightforward to check that in each of linear combinations listed above the sum of absolute values of coefficients does not exceed $n^2$. This proves our claim.

To finish the proof it is enough to see that in our case we have $\alpha^{-1} \leq 2(n+3)$ and $2p \leq n+3$. Moreover $m+2\beta^{2p} \leq n+2 + 2n^{2(n+3)}$ and we can check by hand that
$$n+2 + 2n^{2(n+3)} \leq (n+3)^{2(n+3)},$$
and thus
$$(n+2 + 2n^{2(n+3)})^7 \leq (n+3)^{14(n+3)} \leq (n+3)^{2(n+3)^2}$$
since $n \geq 4$. Therefore a straightforward bound yields
$$\lambda(Y, X) > 1 + (n+3)^{-2(n+3)}  \left ((n+3)^{-6} 2^{-16} (n+3)^{-3} (n+3)^{-11} (n+3)^{-4} \right )^{6(n+3)^2}$$
$$>1 + \left (8 \left ( n + 3 \right )^{5} \right )^{-30(n+3)^2},$$
for an arbitrary hyperplane $Y \subset X$ and the conclusion follows.

\qed

\section{Concluding remarks}
\label{sekcjauwagi}

In the preceeding sections we established a quantitative lower bound on relative projection constant for hyperplanes of subspaces of $\ell_{2p}^m$ spaces. In particular, we proved an existence of an $n$-dimensional normed space which every projection onto hyperplane has norm at least $1 + \left (8 \left ( n + 3 \right )^{5} \right )^{-30(n+3)^2} > 1 + \exp(-C n^2 \log n)$. It is reasonable to conjecture that both of this estimations could be significantly improved.

\begin{problem}
\label{problem1}
Improve lower bound given in Theorem \ref{glowne} for hyperplanes of $\ell_{2p}^m$ spaces. Give any non-trivial estimation in the three-dimensional case.
\end{problem}

Clearly our result can be improved, as in many places we have used some crude bounds and sacrified precision of the estimation for a clarity of the reasoning. We believe however, that with some more efficient ideas it is possible to obtain a lower bound of a much better order.

\begin{problem}
\label{problem2}
Improve lower bound given in Corollary \ref{wniosek}. Give any non-trivial estimation in the three-dimensional case. Is it true that there exists $c>0$ such that for every $n \geq 3$ there exists an $n$-dimensional normed space $X$ satisfying $\lambda(Y, X) > 1+c$ for every hyperplane $Y \subset X$?
\end{problem}

In this problem one can suspect that there is even more room for improvement. We believe that our techniques could be used for a lot of other spaces as well. The two important elements: modulus of convexity of the dual and form of the supporting functional are determined for many classes of normed spaces. It is possible that some better estimate could be obtained for subspaces of some Orlicz-Musielak spaces, which generalize $\ell_p^m$ spaces in a very practical way. Probabilistic constructions also seem to be quite promising way to approach, even if they usually work in the asymptotic setting.

\begin{problem}
Give analogues of Theorem \ref{glowne} and Corollary \ref{wniosek} in the setting of an arbitrary subspace $Y \subset X$ such that $2 \leq \dim Y \leq \dim X - 1$.
\end{problem}

The problem above just rephrases the original question of Bosznay and Garay. We feel that with some additional work, methods presented in the paper could be refined to yield a lower bound for an arbitrary subspace.

We conclude the paper with the problem of discrete geometry originating from Section \ref{sekcjafunkcjonaly}.

\begin{problem}
Let $m,\ n \geq 1$ be integers. Consider a norm $|| \cdot || \in \mathbb{R}^n$ and collection of normed linear functionals $f_1, f_2, \ldots, f_m: \mathbb{R}^n \to \mathbb{R}$. Provide some estimates of $\max_{||x||=1} \min_{1 \leq i \leq m} |f_i(x)|$.
\end{problem}

The problem is formulated in a general way but we can propose some specific variations, all of them seeming to be non-trivial. First of all we can fix the norm $|| \cdot ||$ to be specific (for example some $\ell_p$-norm) and ask for a best possible lower bound on the considered quantity. Usually it will be probably extremely hard to give a closed formula for arbitrary $m,\ n$ but here again we have some possibilities. For example, we can fix $m$ and let $n \to \infty$ and determine the asymptotics. Or vice versa. Perhaps even in the cases of small $m$ and arbitrary $n$ the problem can be challenging. Moreover, we can let norm $|| \cdot ||$ not to be fixed and and try to find best possible lower bound for an arbitrary norm. Here again we have different possibilites for $m$ and $n$.

Some of the proposed variations may have been already considered in the literature, but it seems that problems of this kind can make an interesting and broad area of further research.

\end{document}